\documentclass[reqno,b5paper]{amsart}

\usepackage{amsmath}
\usepackage{amssymb}
\usepackage{amsthm}
\usepackage{enumerate}
\usepackage[mathscr]{eucal}
\usepackage{eqlist}

\theoremstyle{plain}
\newtheorem{theorem}{Theorem}[section]
\newtheorem{prop}[theorem]{Proposition}

\theoremstyle{definition}
\newtheorem{definition}{Definition}[section]
\newtheorem{remark}{\textnormal{\textbf{Remark}}}

\theoremstyle{remark}
\newtheorem{example}{Example}

\numberwithin{equation}{section}

\begin{document}

\title[1. The involutiveness and standard basis]%
{On the internal approach to differential equations\\ 1. The involutiveness and standard basis}
\author[Veronika Chrastinov\'a \and V\'aclav Tryhuk]%
{Veronika Chrastinov\'a \and V\'aclav Tryhuk}

\newcommand{\acr}{\newline\indent}

\address{Brno University of Technology\acr
Faculty of Civil Engineering\acr
Department of Mathematics\acr
Veve\v{r}\'{\i} 331/95, 602 00 Brno\acr
Czech Republic}

\email{chrastinova.v@fce.vutbr.cz, tryhuk.v@fce.vutbr.cz}

\thanks{This paper was elaborated with the financial support of the European
Union's "Operational Programme Research and Development for
Innovations", No. CZ.1.05/2.1.00/03.0097, as an activity of the
regional Centre AdMaS "Advanced Materials, Structures and
Technologies".}



\begin{abstract}
The article treats the geometrical theory of partial differential equations in the absolute sense, i.e., without any additional structures and especially without any preferred choice of independent and dependent variables. The equations are subject to arbitrary transformations of variables in the widest possible sense. In this preparatory Part 1, the involutivity and the related standard bases are investigated as a~technical tool within the framework of commutative algebra. The particular case of ordinary differential equations is briefly mentioned in order to demonstrate the strength of this approach in the study of the structure, symmetries and constrained variational integrals under the simplifying condition of one independent variable. In full generality, these topics will be investigated in subsequent Parts of this article.
\end{abstract}

\maketitle

\section{Preface}\label{sec1}
The \emph{internal} (equivalently: \emph{absolute}) \emph{theory} of differential equations indifferent to the actual choice of dependent and independent variables in the widest possible sense was latently initiated in Lie\rq{}s approach to the first--order partial differential equations and contact transformations. However, it was implicitly proclaimed in full generality only by E.~Cartan in the pseudogroup theory \cite{T1} and then explicitly in his later article devoted to the Monge problem \cite{T2}. Alas, this idea was found difficult and the lack of further results with convincing outcomes did not sufficiently stimulate the following developments  for a~long time. The subsequent investigations by Janet, Ritt and Kolchin, Goldschmidt and Steinberg, Manin, Kuperschmidt, Lychagin and Vinogradov, Pommaret, Olver, Gardner, Stormark, Kuranishi, Kamran, Anderson and Fels (to name just a~few) were governed by other conception: by the study of special classes of differential equations firmly localized in finite--order fibered jet spaces (the \emph{external theory}) together with rather sophistical tools of differential algebra and rigid order--preserving $G$--structures. Some elements of the absolute theory nevertheless occur in the last part of monograph \cite{T3}, the free choice of dependent variables appears under the name of \emph{differential substitutions} in certain studies on integrable equations \cite{T4} and the \emph{generalized} (or: \emph{higher--order, Lie--B\"acklund}) \emph{infinitesimal symmetries} cannot be understood without the use of some modest internal concepts \cite{T5}.

For better clarity, let us recall the original E.~Cartan's idea. Two classical systems $S$ and $S\rq{}$ of differential equations are called \emph{absolutely equivalent} if there exist prolongations $\sum$ of $S$ and $\sum \rq{}$ of $S\rq{}$ defined on certain spaces $M, M\rq{}$ of equal finite dimensions such that appropriate invertible mapping $M\rightarrow M\rq{}$ identifies $\sum$ with $\sum\rq{}.$ It should be noted that E.~Cartan's prolongation is a~very broad concept defined as follows. System $\sum$ on a~space $M$ is a~prolongation of a~system $S$ defined on certain space $N$ if the variables of $M$ involve variables of $N$ (i.e., $N$ is a~factorspace of $M$) which moreover provides natural bijection between solutions of $\sum$ and $S$ (by using the projection of $M$ onto $N$). It follows that there are \emph{many} prolongations $\sum$ of $S$ in this sense. For instance, the system \[(\Delta u=\,)\ \frac{\partial^2u}{\partial x^2}+\frac{\partial^2u}{\partial y^2}=0,\ \frac{\partial u}{\partial x}=v,\ u\frac{\partial u}{\partial y}+\frac{\partial^3u}{\partial y^3}=w\] in the space $M$ of variables $x,y,u,v,w$ is a~prolongation of the first equation $\Delta u=0$ in the space $N$ of variables $x,y,u.$

We undertake this idea with the following adjustment. Let $\Omega$ and $\Omega'$ be infinite (better: the largest possible) prolongations of $S$ and $S'$ defined on spaces $\mathbf M$ and $\mathbf M',$ respectively. Since $\Omega$ and $\Omega'$ may be considered as a~prolongation of $\sum$ and $\sum'$ as well, we conclude that $S$ and $S'$ are absolutely equivalent in the above E.~Cartan's sense if and only if appropriate invertible mapping $\mathbf M\rightarrow\mathbf M'$ identifies $\Omega$ with $\Omega'.$ So the use of \emph{uncertain} prolongations $\sum, \sum'$ is deleted but the main advantage of the modified approach lies in the fact that the infinite prolongations $\Omega, \Omega'$ can be characterized without any use of coordinates \cite{T6}. As a~result, the "absolute nature" of this approach is automatically ensured. We speak of \emph{diffieties} $\Omega.$ It should be noted that a~given diffiety $\Omega$ may be regarded as the infinite prolongation of \emph{many} rather dissimilar systems $S$ of differential equations according to the additional choice of dependent and independent variables which is regarded as a~mere technical tool.

The particular case of \emph{ordinary} differential equations was already treated in previous article \cite{T7} and the main achievements are briefly recalled in Section~\ref{sec3} below for the convenience of the reader. The achievements rest on the method of \emph{standard bases} analogous to the common contact forms of the jet theory. Roughly saying, quite arbitrary \emph{system of ordinary differential equations} is "identified" with jet theory of \emph{all curves}, i.e., with the trivial (empty) system by using just a~standard basis. In the general case of \emph{partial} differential equations, this is more difficult task and we need involutiveness appropriately expressed in terms of commutative algebra as a~preparatory tool. This is the central topic of this Part 1 and we believe that the subsequent Parts will be more interesting.

In this article, we do not need any advanced technical tools. Only the most fundamental properties of differential forms and vector fields are enough. Certain novelty lies in the use of the infinite--dimensional underlying spaces $\mathbf M$ of diffieties $\Omega,$ however, all functions to appear are of the classical nature. They depend only on a~finite number of coordinates so that the usual rules of calculations are preserved. In accordance with the common practice, our reasonings are carried out in the local $C^\infty$--smooth category. For instance, our notational convention for a~map $\mathbf M\rightarrow\mathbf M$ allow the domain of definition to be a~proper open subset of~$\mathbf M.$
\section{Fundamental concepts}\label{sec2}
Let $\mathbf M$ be an~infinite--dimensional smooth manifold \emph{modelled on} $\mathbb R^\infty.$ 
In more detail, the space $\mathbf M$ is equipped with (local) coordinates $h^i:\mathbf M\rightarrow\mathbb R$ $(i=1,2,\ldots)$ together with the structural ring $\mathcal F=\mathcal F(\mathbf M)$ (the abbreviation occasionally omitting the letter $\mathbf M$) of real--valued smooth functions $f:\mathbf M\rightarrow\mathbb R$ where $f=f(h^1,\ldots,h^{m(f)})$ in terms of coordinates.
 We consider mappings $\mathbf m:\mathbf M\rightarrow\mathbf M$ given by certain formulae 
\[\mathbf m^*h^i=H^i\qquad (Hî^i=H^i(h^1,\ldots,h^{m(i)})\in\mathcal F; i=1,2,\ldots\,).\]
Analogous invertible formulae describe the admissible change of coordinates, see also \cite{T6, T7}.

Let $\Phi=\Phi(\mathbf M)$ be the $\mathcal F$--module of differential 1--forms
$$\varphi=\sum f^i\mbox{d}g^i\qquad (\mbox{finite sum; } f^i,g^i\in\mathcal F).$$
The familiar rules of exterior differential analysis can be applied without change. In particular, the form $\mathbf m^*\varphi=\sum\mathbf m^*f^i\mbox{d}\mathbf m^*g^i\in\Phi$ makes good sense.
Let $\mathcal T=\mathcal T(\mathbf M)$ be the $\mathcal F$--module of vector fields
$$Z=\sum z^i\dfrac{\partial}{\partial h^i}\qquad (\mbox{infinite sum; }z^i\in\mathcal F \text{ are arbitrary})$$
in terms of coordinates. In coordinate--free manner, the vector field $Z$ is interpreted as the $\mathcal F$--linear function on the $\mathcal F$--module $\Phi$ determined by the duality pairing
$$\mbox{d}h^i(Z)=Z\rfloor\mbox{d}h^i=Zh^i=z^i\in\mathcal F\qquad (i=1,2,\ldots\,).$$
With this principle in mind, let certain forms $\varphi^1,\varphi^2,\ldots\in\Phi$ provide a~\emph{basis of} $\Phi$ such that every $\varphi\in\Phi$ admits a~unique representation $\varphi=\sum f^i\varphi^i $ (finite sum, $f^i\in\mathcal F$). Then the values
\begin{equation}\label{eq2.1}\varphi^i(Z)=Z\rfloor\varphi^i=\bar z^i\in\mathcal F\quad (i=1,2,\ldots\,)\end{equation}
uniquely determine the vector field which we denote by
\begin{equation}\label{eq2.2}Z=\sum \bar z^i\dfrac{\partial}{\partial\varphi^i}\quad (\mbox{infinite sum}, \bar z^i\in\mathcal F) \ \end{equation} (abbreviations like $\partial/\partial f=\partial/\partial\mbox{d}f$ for the notation will appear).
In particular, the vector fields $\partial/\partial\varphi^1,\partial/\partial\varphi^2,\ldots\,\in\mathcal T$ provide the~\emph{weak basis} of module $\mathcal T$ \emph{dual} to the original basis $\varphi^1,\varphi^2,\ldots\,$ of $\Phi.$ 
We recall the \emph{Lie derivative} $\mathcal L_Z=Z\rfloor\mbox{d}+\mbox{d}Z\rfloor$ acting on differential forms and the \emph{Lie bracket} $\mathcal L_XY=[X,Y]=XY-YX$ $(X,Y\in\mathcal T)$ without comment.

We shall deal with various $\mathcal F$--submodules $\Omega\subset\Phi$ of differential forms together with their \emph{orthogonal submodules} $\Omega^\perp\subset\mathcal T.$  This module $\Omega^\perp$ orthogonal to $\Omega$ includes vector fields $Z\in\mathcal T$ satisfying $\omega(Z)=0$ for all $\omega\in\Omega.$ The (local) existence of bases in various submodules $\Omega\subset\Phi$ to appear is tacitly postulated.

\begin{definition}
\label{def2.1}\rm
A~finite--codimensional submodule $\Omega\subset\Phi(\mathbf M)$ is called a~\emph{diffiety} on $\mathbf M$ if there exists filtration
$\Omega_*:\Omega_0\subset\Omega_1\subset\cdots\subset\Omega=\cup\Omega_l $ by finite--dimensional submodules $\Omega_l\subset\Omega$ $(l=0,1,\ldots\,)$ such that
\begin{equation}\label{eq2.3}\mathcal L_{\mathcal H}\Omega_l\subset\Omega_{l+1}\quad (\mbox{all }l),\quad \Omega_l+\mathcal L_{\mathcal H}\Omega_l=\Omega_{l+1}\quad (l\mbox{ large enough}),\end{equation}
the so--called \emph{good filtration}.
We systematically denote $\mathcal H=\mathcal H(\Omega)=\Omega^\perp$ if $\Omega$ is a~diffiety from now on. For every submodule $\varTheta\subset\Phi(\mathbf M),$ the submodule $\mathcal L_{\mathcal H}\varTheta\subset\Phi(\mathbf M)$ is generated by all differential forms $\mathcal L_Z\vartheta$ $(Z\in\mathcal H, \vartheta\in\varTheta).$
\end{definition}

\begin{remark}\label{rem2.1}
In succinct terms, diffiety $\Omega$ is such a~\emph{finite--codimensional} submodule $\Omega\subset\Phi(\mathbf M)$ that we have
\begin{equation}\label{eq2.4}\Omega=\Gamma+\mathcal L_{\mathcal H}\Gamma+\mathcal L^2_{\mathcal H}\Gamma+\cdots
\end{equation} for an~appropriate \emph{finite--dimensional} submodule $\Gamma\subset\Phi(\mathbf M)$ satisfying $\mathcal H\subset\Gamma^\perp;$ choose $\Gamma=\Omega_l$ with $l$ large enough. We shall later see that the codimensionality assumption can be in a~certain sense omitted. There are many filtrations $\Omega_*$ of diffiety $\Omega$ with properties (\ref{eq2.3}). In particular, we notice the $c$--\emph{lift} (fixed $c=0,1,\ldots\,$) denoted $\Omega_{*+c}$ where the lower terms $\Omega_0,\ldots,\Omega_{c-1}$ of the original filtration $\Omega_*$ are neglected. We shall however need quite opposite arrangements of a~given good filtration $\Omega_*$ later on.
\end{remark}

\begin{definition}
\label{def2.2}\rm
Let $\Omega\subset\Phi(\mathbf M)$ be a~diffiety of codimension $n=n(\Omega)\geq 1.$  Functions $x_1,\ldots,x_n\in\mathcal F(\mathbf M)$ are called \emph{independent variables} for $\Omega$ if the differentials $\mbox{d}x_1,\ldots,\mbox{d}x_n$  are linearly independent modulo $\Omega.$ Alternatively saying, there exists unique representation
\[\varphi=\sum f_i\mbox{d}x_i+\omega\quad (f_i\in\mathcal F(\mathbf M), \omega\in\Omega)\]
for every $\varphi\in\Phi(\mathbf M).$ Vector fields $D_1,\ldots,D_n\in\mathcal T(\mathbf M)$  uniquely defined by
\[\varphi(D_i)=f_i, \omega(D_i)=0\quad (i=1,\ldots,n)\]
are called \emph{formal} (or: \emph{total}) \emph{derivatives} to the diffiety $\Omega.$ They provide a~useful but not always the best possible basis of module $\mathcal H(\Omega).$
\end{definition}

\medskip

At this place, we can clarify the interrelations between the classical differential equations and diffieties. In more detail they are as follows.

In one direction, a~given system of differential equations may be represented as a~Pfaffian system $\omega=0$ and the module $\Omega$ generated by all such 1--forms $\omega$  is just the diffiety $\Omega.$ More precisely, we must deal with the infinite prolongation of the Pfaffian system. Particular examples to follow later on will be easy in this respect and do not need any comments here.We also refer to quite instructive Section~\ref{sec3} below.

 In the reverse direction, let us consider some diffiety $\Omega\subset\Phi(\mathbf M).$ Due to the existence of filtration $\Omega_*$ with properties (\ref{eq2.3}), there exist forms $\gamma^j=\sum a^j_k\mbox{d}h^k\in\Omega$ $(j=1,\ldots,m)$ such that the forms $$\gamma^j,\gamma^j_i=\mathcal L_{D_i}\gamma^j,\gamma^j_{ii'}=\mathcal L_{D_i}\mathcal L_{D_{i'}}\gamma^j,\ldots\qquad (i,i',\ldots=1,\ldots,n;\, j=1,\ldots,m)$$
generate the module $\Omega,$ see Remark~\ref{rem2.1}. The corresponding Pfaffian system $\gamma^j=\gamma^j_i=\gamma^j_{ii'}=\cdots =0$ is obviously equivalent to the differential equations
$$\sum a^j_k\frac{\partial h^k}{\partial x_i}=0, \frac{\mbox{d}}{\mbox{d} x_{i'}}\sum a^j_k\frac{\partial h^k}{\partial x_i}=0,\frac{\mbox{d}^2}{\mbox{d} x_{i'}\mbox{d} x_{i''}}\sum a^j_k\frac{\partial h^k}{\partial x_i}=0,\cdots$$
for a~finite number of unknown functions $h^k=h^k(x_1,\ldots,x_n)$ occuring in the forms $\gamma^1,\ldots,\gamma^m.$
So we have the infinite prolongation in the classical sense.
\begin{definition}\label{def2.3}\rm
Let $\Omega\subset\Phi(\mathbf M)$ be a~diffiety. Invertible mapping $\mathbf m:\mathbf M\rightarrow\mathbf M$ is called \emph{automorphism} (or: \emph{symmetry}) of $\Omega$ if $\mathbf m^*\Omega\subset\Omega.$ Vector field $Z\in\mathcal T(\mathbf M)$ is called a~\emph{variation} of $\Omega$ if $\mathcal L_Z\Omega\subset\Omega.$ If $Z$ moreover (locally) generates a~one--parameter group of transformations $\mathbf m(t):\mathbf M\rightarrow\mathbf M$ $(-\varepsilon<t<\varepsilon,\, \varepsilon >0)$ then $Z$ is called the \emph{infinitesimal symmetry} of $\Omega.$
\end{definition}

We intentionally introduce rather narrow definition at this place in full accordance with the common practice, however, the invertibility of $\mathbf m$ is lacking in certain situations to appear later on and then we speak of a~\emph{morphism} $\mathbf m$. The term of a~"variation" is also unorthodox but well--founded by the calculus of variations, see \cite[Section~7]{T7}. Recall that only very special vector fields $Z\in\mathcal T(\mathbf M)$ generate a~true one--parameter group \cite{T6, T7, T8} and it follows that the use of the terms like "Lie--B\"acklund" or "generalized" infinitesimal symmetry is misleading to denote \emph{every} vector field $Z$ satisfying only the weak condition $\mathcal L_Z\Omega\subset\Omega.$

\begin{remark}\label{rem2.3}
Definitions \ref{def2.1}--\ref{def2.3} make a~good sense even if $\mathbf M$ is a~finite--dimensional space and then they concern the "completely integrable" Pfaffian system $\Omega$ where the Frobenius theorem can be applied and $\Omega$ admits a~basis consisting of total differentials. Even the strange subcase $n=n(\Omega)=0$ hence $\Omega=\Phi(\mathbf M), \mathcal H=0$ may be formally useful in certain respect.
\end{remark}

We separately mention the particular case $n=n(\Omega)=1$ of the underdetermined systems of ordinary differential equations where the general theory simplifies and some results \cite{T7} can be easily referred to. We believe that then the general case $n>1$ becomes more reliable for the reader.
\section{Deviation to one independent variable}\label{sec3}
Passing to the particular case $n=n(\Omega)=1$ of diffieties $\Omega\subset\Phi(\mathbf M),$ we abbreviate by $x=x_1$ the independent variable and $D=D_1$ the total derivative. Let us recall that $\mathcal L_D\Omega\subset\Omega$ and there exist $\gamma^1,\ldots,\gamma^m\in\Omega$ such that the family of all forms $\mathcal L_D^s\gamma^j$ $(j=1,\ldots,m;\,s=0,1,\ldots)$ together with differential $\mbox{d}x\notin\Omega$ generate the module $\Phi(\mathbf M)$ of all differential 1--forms on the underlying space $\mathbf M.$

Clearly $\varphi\in\Phi$ is lying in $\Omega$ if and only if $\varphi(D)=0$ therefore
\[\mbox{d}f-Df\cdot\mbox{d}x,\ Df\cdot\mbox{d}g-Dg\cdot\mbox{d}f\in\Omega\qquad (f,g\in\mathcal F).\]
So we have many forms lying in $\Omega$ and even the bases of $\Omega$ can be easily found in current examples, however, the most interesting result is as follows \cite{T6, T7}.
There exist forms $\tau^1,\ldots,\tau^K,\pi^1,\ldots,\pi^\mu\in\Omega$ $(K=K(\Omega),\,\mu=\mu(\Omega))$ such that
\begin{equation}\label{eq3.1}\tau^k,\, \pi^j_s=\mathcal L^s_D\pi^j\qquad (k=1,\ldots, K;\,j=1,\ldots,\mu;\,s=0,1,\ldots \,)\end{equation}
is a~basis of $\Omega$, the so--called \emph{standard basis}, where moreover
\begin{equation}\label{eq3.2}\mathcal L_D\tau^k\sim 0,\mbox{d}\tau^k\sim 0\quad (\text{mod }\tau^1,\ldots,\tau^K),\quad \mbox{d}\pi^j_s\sim \mbox{d}x\wedge\pi^j_{s+1}\quad (\text{mod }\Omega\wedge\Omega).\end{equation}
The standard bases of any diffiety $\Omega$ are not unique and can be determined using the tools of a~mere basic linear algebra. They are  useful in applications as follows.

\medskip

\emph{First}. Let $\mathcal R(\Omega)\subset\Omega$ be the~submodule generated by all differentials $\mbox{d}f$ $(f\in\mathcal F)$ that are lying in $\Omega.$ Then $\tau^1\ldots,\tau^ K$ is a~basis of $\mathcal R(\Omega)$ and there exists alternative basis $\mbox{d}t^1,\ldots,\mbox{d}t^K$ consisting of differentials, see (\ref{eq3.2}) and apply the Frobenius theorem. We may introduce space $\mathbf M^0$ (locally $\mathbf M^0=\mathbb R^{K+1})$ with coordinates $x,t^1,\ldots,t^K$ and diffiety $\Omega^0\subset\Phi(\mathbf M^0)$ with the basis $\text{d}t^1,\ldots,\text{d}t^K.$ The space $\mathbf M^0$ is a~factorspace of $\mathbf M$ and $\Omega^0$ may be regarded as a~diffiety "induced" by the primary diffiety $\Omega.$ We have  the "compositions series" $\Omega^0\subset\Omega^1=\Omega$ where $\Omega^0=0$ is trivial just in the \emph{controllable case} $K=0.$

\medskip

\emph{Second}. Let $\mathbf m: \mathbf M\rightarrow\mathbf M$ be a~symmetry of $\Omega.$ Clearly $\mathbf m^*\mathcal R(\Omega)\subset\mathcal R(\Omega)$ and thanks to the \emph{prolongation rule} \cite[Lemma 5.1]{T7}
\begin{equation}\label{eq3.3}\mathbf m^*x\cdot\mathbf m^*\mathcal L_D\omega=\mathcal L_D\mathbf m^*\omega\qquad (\omega\in\Omega), \end{equation}
already the forms $\mathbf m^*\pi^1,\ldots,\mathbf m^*\pi^\mu$ and the factor $\mathbf m^*x$ uniquely determine the remaining forms $\mathbf m^*\pi^j_s$ of the basis (\ref{eq3.1}) hence all forms $\mathbf m^*\omega$ $(\omega\in\Omega).$
At this algebraical level, the forms $\mathbf m^*\pi^k$ and the factor $\mathbf m^*x$ can be arbitrarily chosen to a~large extent: in order to ensure the invertibility of $\mathbf m,$ the equality $\mathbf m^*\mathcal R(\Omega)=\mathcal R(\Omega)$ and the conditions $\pi^j\in\mathbf m^*\Omega$ $(j=1,\ldots,\mu)$ are enough \cite[Lemma 5.3]{T7}.

\medskip

\emph{Third}. Let $Z\in\mathcal T(\mathbf M)$ be a~variation of $\Omega.$ Clearly $\mathcal L_Z\mathcal R(\Omega)\subset\mathcal R(\Omega)$  and thanks to the \emph{prolongation rule} \cite[Lemma 5.4]{T7}
\begin{equation} \label{eq3.4} (\mathcal L_D\omega)(Z)=D\omega(Z)\qquad (\omega\in\Omega), \end{equation}
we may choose arbitrary values
\[z=Zx, p^j=\pi^j(Z)\in\mathcal F\ \  (j=1,\ldots,\mu),\   z^k=z^k(t^1,\ldots,t^K)\ \ (k=1,\ldots,K)\]
and \emph{all variations} 
\begin{equation}\label{eq3.5}
Z=z\frac{\partial}{\partial x}+\sum z^k\frac{\partial}{\partial t^k}+\sum D^sp^j\frac{\partial}{\partial \pi^j_s}
 \end{equation}
\emph{are obtained in explicit terms.}
We recall on this occasion that $Z$ generates a~true one--parameter group of symmetries if and only if  all forms $\mathcal L_Z^s\pi^j$ $(j=1,\ldots,\mu;\,s=0,1,\ldots\,)$ are contained in a~finite--dimensional module, see \cite{T8} and \cite[Theorem 5.2]{ T7}. This achievement provides effective algorithm for the calculation of the higher--order infinitesimal symmetries of a~given diffiety.

\medskip

\emph{Fourth}. The \emph{Lagrange problem} of the calculus of variations appears if together with a~diffiety $\Omega\subset\Phi(\mathbf M)$ representing  the \emph{differential constraints}, also a~form $\varphi\in\Phi(\mathbf M)$ representing \emph{variational integral} $\int\varphi\ $ is given \cite{T6, T7}.
In this "absolute" variant of the calculus of variations, appropriate use of  the standard basis provides the Euler--Lagrange system and the Poincar\'e--Cartan form within the framework of the space $\mathbf M,$ without any use of the common Lagrange multipliers.

\medskip

All these wel--known achievements \cite{T7} will be adapted for the general case of partial differential equations in future.
We moreover believe that fruitful interrelations to the general theory of Lie--B\"acklund and Darboux transformations, non--local symmetries, Lie--Cartan pseudogroups and the theory of the variational bicomplex should be expected.

Let us conclude with simple examples.
\begin{example}\label{Ex1} (\emph{The contact diffiety}\/).
We introduce space $\mathbf M(m)$ that locally admits the familiar \emph{jet coordinates}
\begin{equation}\label{eq3.6}x,w^j_s\qquad (j=1,\ldots,m;\,s=0,1,\ldots\,)  \end{equation}
 and diffiety $\Omega(m)\subset\Phi(\mathbf M(m))$ with the basis consisting of classical \emph{contact forms}
\begin{equation}\label{eq3.7}
\omega^j_s=\mbox{d}w^j_s-w^j_{s+1}\mbox{d}x\qquad (j=1,\ldots,m;\,s=0,,\ldots\,) .
\end{equation}
This is (locally) the well--known infinite--order jet space of $x$--parametrized curves in $\mathbb R^{m+1},$ however, coordinates (\ref{eq3.6}) are regarded as a~mere technical tool here. \end{example}
Clearly
\begin{equation}\label{eq3.8}
D=\frac{\partial}{\partial x}+\sum w^j_{s+1}\frac{\partial}{\partial w^j_s}\in\mathcal H(\Omega(m)),\ \mathcal L_D\omega^j_s=\omega^j_{s+1}\ \end{equation}
$(j=1,\ldots,m;\,s=0,1,\ldots\,)$
so we have the standard basis (\ref{eq3.1}) where $K=K(\Omega)=0,$ $\mu=\mu(\Omega)=m$ and $ \pi^j_s=\omega^j_s$. Let us note on this occasion that the class of all these diffieties $\Omega(m)$ corresponding to classical "empty systems" of ordinary differential equations was not yet characterized  in coordinate--free terms if $m>1$, this is the ancient and rather difficult \emph{Monge problem} \cite{T6}.

There is the natural "order preserving" filtration $\Omega(m)_*$ where the forms (\ref{eq3.7}) with restriction $s\leq l$ generate the $l$--th order term $\Omega(m)_l.$ Symmetries $\mathbf m$ of the contact diffiety $\Omega(m)$ need not in general preserve this natural filtration. Three cases should be distinguished as follows.
\begin{center}\begin{picture}(115,100)(,-30)
\put(10,50){\line(1,0){10}}\put(20,50){\line(0,-1){55}}\put(5,55){$\Omega_0$}
\qbezier(20,40)(-10,25)(20,10)\put(18,11){\vector(2,-1){1}}
\put(40,50){\line(1,0){10}}\put(50,50){\line(0,-1){55}}\put(40,55){$\Omega_1$}
\qbezier(50,40)(20,25)(50,10)\put(48,11){\vector(2,-1){1}}
\put(70,50){\line(1,0){10}}\put(80,50){\line(0,-1){53}}\put(70,55){$\Omega_2$}
\qbezier(80,40)(50,25)(80,10)\put(78,11){\vector(2,-1){1}}
\put(92,55){$\ldots$}\put(5,-25){point symmetries}
\end{picture}
\begin{picture}(115,100)(0,-30)
\put(10,50){\line(1,0){10}}\put(20,50){\line(0,-1){55}}\put(5,55){$\Omega_0$}
\qbezier[50](10,10)(44,29)(78,48)\qbezier(60,38)(57,0)(30,21)\put(32,20){\vector(-1,1){1}}
\put(40,50){\line(1,0){10}}\put(50,50){\line(0,-1){55}}\put(40,55){$\Omega_1$}
\qbezier[50](30,-5)(66,15)(96,33)\qbezier(85,26)(82,-12)(55,9)\put(57,8){\vector(-1,1){1}}
\put(70,50){\line(1,0){10}}\put(80,50){\line(0,-1){53}}\put(70,55){$\Omega_2$}
\put(92,55){$\ldots$}
\put(-18,-25){general group of symmetries}
\put(30,-45){Figure 1.}
\end{picture}
\begin{picture}(100,100)(0,-30)
\put(10,50){\line(1,0){10}}\put(20,50){\line(0,-1){55}}\put(5,55){$\Omega_0$}
\qbezier[30](60,45)(83,35)(106,26)\qbezier(80,20)(90,20)(97,29)\put(96,28){\vector(1,1){1}}
\put(40,50){\line(1,0){10}}\put(50,50){\line(0,-1){55}}\put(40,55){$\Omega_1$}
\qbezier[60](10,45)(55,25)(100,5)\qbezier(50,10)(60,10)(67,19)\put(66,18){\vector(1,1){1}}
\put(70,50){\line(1,0){10}}\put(80,50){\line(0,-1){53}}\put(70,55){$\Omega_2$}
\qbezier[50](10,25)(40,10)(70,-5)\qbezier(20,0)(31,0)(38,10)\put(37,9){\vector(1,1){1}}
\put(92,55){$\ldots$}\put(95,55){$\ldots$}\put(10,-25){general symmetries}
\end{picture}
\end{center}

\vspace{0.7cm} 

The \emph{left--hand schema} describes the classical order--preserving symmetries. In more generality, if $\mathbf m$ is a~symmetry such that $\mathbf m^*\Omega_L(m)\subset\Omega_L(m)$ (fixed $L$) then $\mathbf m^*\Omega_l(m)=\Omega_l(m)$ for all $l$ and we have either a~\emph{point symmetry} (if $m>1$) or the \emph{Lie's contact transformation symmetry} (if $m=1$). This is the familiar \emph{Lie--B\"acklund theorem}. We may refer to \cite{T9} for a~short tricky proof. In actual literature, differential equations are as a~rule considered just in the finite--order jet spaces. The remaining higher--order symmetries cause many difficulties since the localization of the dotted lines in Figure~1 is not known in advance. Moreover there are two quite dissimilar possibilities. 
 The \emph{middle schema} describes such symmetries which may be included into a~Lie group. They preserve certain finite--dimensional submodules of $\Omega(m)$ and can be determined by the moving frame method.
The \emph{right--hand schema} describes the most general symmetries where both the Lie's method of the infinitesimal transformations and the original E. Cartan's general equivalence method fail. 

In terms of coordinates (\ref{eq3.6}), a~symmetry $\mathbf m:\mathbf M\rightarrow\mathbf M$ of diffiety $\Omega(m)$ is given by certain formulae
\begin{equation}\label{eq3.9}
\mathbf m^*x=F,\ \mathbf m^*w^j_s=F^j_s\qquad (F^j_{s+1}=\frac{DF^j_s}{DF};\, j=1,\ldots,m;\, s=0,1,\ldots\,)
\end{equation}
where $F,F^j_s\in\mathcal F(\mathbf M(m))$ and we suppose $DF\neq 0$ (hence $F\neq const.$). The recurrence is equivalent to the inclusions $\mathbf m^*\omega^j_s\in\Omega(m).$ We recall that the invertibility of $\mathbf m$ is ensured if $\Omega(m)_0\subset\mathbf m^*\Omega(m),$ see also \cite[Theorem~5.1 or Lemma~5.3]{T7}. We will not discuss the classical left--hand case here. Instead, we briefly mention the special "wave" method \cite{T11} in order to illustrate the middle and the right--hand cases of Figure~1. It is rather interesting that the inverse $\mathbf m^{-1}$ will be explicitly found by using the "reverse" wave, however, the original Huygens principle fails.
\begin{example}\label{Ex2}
Continuing with the above contact diffiety $\Omega(m),$ let us moreover introduce the "duplicate" $\bar\Omega(m)\subset\Phi(\bar{\mathbf M}(m)).$ Then the coordinates (\ref{eq3.6}) and the contact forms (\ref{eq3.7}) are completed with bars:
\[\bar x,\bar w^j_s, \bar\omega^j_s=\mbox{d}\bar w^j_s-\bar w^j_{s+1}\mbox{d}\bar x\qquad (j=1,\ldots,m;\, s=0,1,\ldots)\] and we have the total derivative
\[\bar D=\frac{\partial}{\partial\bar x}+\sum\bar w^j_{s+1}\frac{\partial}{\partial\bar w^j_s}\in\mathcal H(\bar\Omega(m)).\]
With this preparation, let moreover
\[W^j=W^j(x,w^1_0,\ldots,w^m_0,\bar x,\bar w^1_0,\ldots,\bar w^m_0)\qquad (j=1,\ldots,m)\]
be given smooth functions. Then our proposition is as follows.
\end{example}
\begin{prop}\label{ Proposition3.1}
Assume that the system \mbox{$W^1=\cdots =W^m=DW^1=0$} admits a~unique solution
\begin{equation}\label{eq3.10}
\bar x=F,\ \bar w^j_0=F^j_0\qquad(F,F^j_0\in\mathcal F(\mathbf M(m));\, j=1,\ldots,m)\end{equation}
such that $DF\neq 0$ by virtue of the classical implicit function theorem. Let analogously the system $W^1=\cdots =W^m=\bar DW^1=0$
admits a~unique solution
\begin{equation}\label{eq3.11}
x=\bar F,\ w^j_0=\bar F^j_0\qquad(\bar F,\bar F^j_0\in\mathcal F(\mathbf{\bar M}(m));\, j=1,\ldots,m)\end{equation}
such that $\bar D\bar F\neq 0.$ Then, if the remaining functions $F^j_s,\bar F^j_s$ $(s>0)$ are defined by the recurrence occuring in $(\ref{eq3.9})$, we obtain certain automorphism $(\ref{eq3.9})$ by using $(\ref{eq3.10})$ and moreover its inverse by using $(\ref{eq3.11}).$ That is, formal change of notation $\bar x=\mathbf m^*x, \bar w^j_s=\mathbf m^*w^j_s$ $(\ref{eq3.10})$ and $(\ref{eq3.11})$ provide explicit formulae for the symmetry $\mathbf m$ and its inverse $\mathbf m^{-1}.$
\end{prop}
\begin{proof} Let the equations $W^1=\cdots =W^m=DW^1=0$ be resolved by (\ref{eq3.10}). We may introduce additional functions $F^j_s\ (s>0)$ by recurrence. Then
\[0=\mbox{d}W^1=DW^1\mbox{d}x+\sum\frac{\partial W^1}{\partial w^j_0}\omega^j_0+\bar DW^1\mbox{d}\bar x+\sum\frac{\partial W^1}{\partial {\bar w}^j_0}\bar \omega^j_0\]
identically holds true where moreover $DW^1=0.$ The recurrence implies
\[\bar\omega^j_0=\mbox{d}F^j_0-F^j_1\mbox{d}F\cong DF^j_0\mbox{d}x-F^j_1DF\mbox{d}x\sim 0\quad (\mbox{mod }\Omega(m))\]
and moreover trivially
\[\bar DW^1\mbox{d}\bar x=\bar DW^1\mbox{d}F\cong \bar DW^1\cdot DF\mbox{d}x\quad (\text{mod }\Omega(m)).\]
Altogether $0=\mbox{d}W^1\cong\bar DW^1\cdot DF \mbox{d}x$ $(\text{mod }\Omega(m))$ where $DF\neq 0$ and it follows that $\bar DW^1=0.$
We conclude that equations (\ref{eq3.10}) imply (\ref{eq3.11}).
Analogously equations $W^1=\cdots =W^m=\bar DW^1=0$ imply $DW^1=0$ and we are done.
\end{proof}
One can observe that the classical Lie's contact transformations appear if $m=1.$ In more detail, we have only function $W^1=W(x,w^1_0,\bar x,\bar w^1_0)$ and three equations
\[W=0,\ (DW=)\ \frac{\partial W}{\partial x}+w^1_1\frac{\partial W}{\partial w^1_0}=0,\ (\bar DW=)\ \frac{\partial W}{\partial \bar x}+\bar w^1_1\frac{\partial W}{\partial \bar w^1_0}=0\]
determining $\mathbf m$ and $\mathbf m^{-1}.$ Then the obvious identity
\[0=\mbox{d}W=\frac{\partial W}{\partial w^1_0}\omega^1_0+\frac{\partial W}{\partial \bar w^1_0}\bar\omega^1_0\]
provides link to the primary Lie's approach (equation $\mbox{d}w^1_0-w^1_1\mbox{d}x=0$ is preserved) and moreover clearly
\[\bar w^1_1=\frac{\partial W}{\partial\bar x}/\frac{\partial W}{\partial\bar w^1_0}=F^1_1(x,w^1_0,w^1_1)\]
by virtue of the transformation formulae which means that the space of variables $x,w^1_0,w^1_1$ is preserved (the left--hand Figure~1).

Assuming $m>1,$ we mention only the particular choice
\[W^1=x\bar x-w^1_0+\bar w^1_0,\  W^k=\lambda w^k_0-\bar w^k_0\qquad (k=2,\ldots,m;\, 0\neq\lambda\in\mathbb R)\]
which provides the transformation formulae
\[\bar x=w^1_1,\ \bar w^1_0=xw^1_1-w^1_0,\ \bar w^k_0=\lambda w^k_0\qquad (k=2,\ldots,m).\]
This looks like the Lie's contact transformation combined with a~similarity, however, we have the \emph{order--increasing} transformation not of the classical kind. The space of variables $x,w^1_0,w^1_1,w^2_0,\ldots,w^m_0$ is preserved (the middle Figure~1).
\begin{example}\label{Ex3}
Still continuing with $\Omega(m),$ let $W=W(x,w^1_0,\ldots,w^m_0,\bar x,\bar w^1_0,\ldots,\bar w^m_0)$ be given function. Our proposition is as follows.\end{example}
\begin{prop}\label{ Proposition 2}
 Assume that the system \mbox{$W=DW=\cdots=D^mW=0$} admits a~unique solution $(\ref{eq3.10})$ such that  $DF\neq 0$ and the system $W=\bar DW=\cdots=\bar D^mW=0$ admits a~unique solution $(\ref{eq3.11})$ such that  $\bar DF\neq 0.$ Then the conclusion is the same as above.
\end{prop}
\begin{proof}  We mention only the particular case $m=2$ here. The system $W=DW=0$ and the identity $\mbox{d}W=0$ imply $\bar DW=0$ by the same reasons as above (look at the identity $0=\text{d}W$ with $DW=0$). Then $DW=D^2W=0$ and $\mbox{d}DW=0$ imply $\bar DDW=D\bar DW=0.$ Finally $\bar DW=D\bar  DW=0$ and $\mbox{d}\bar DW=0$ imply $\bar D^2W=0.$ Altogether we see that equations $W=DW=D^2W=0$ imply the equations $W=\bar DW=\bar D^2W=0.$ The converse is obvious which concludes the proof.
\end{proof}
The particular choice
\[W=x\bar x-w^1_0\bar w^1_0-\cdots -w^m_0\bar w^m_0\qquad (m>1)\]
provides a~very simple order--increasing symmetry $\mathbf m$ (where $\mathbf m=\mathbf m^{-1}$) not written here which does not preserve any finite--dimensional module (the right--hand Figure~1).

\bigskip

In the concluding examples, we turn to nontrivial differential equations together with variations and infinitesimal symmetries.
\begin{example}\label{Ex4} (\emph{A~resolved problem.}\/)
Let us deal with variations and infinitesimal symmetries $Z$ of a~differential equation $du/dx=F(dv/dx).$ In the common \emph{external theory}, the equation is identified with the subspace $\mathbf M\subset \mathbf M(2)$ defined by the conditions
\[w^1_1=F, w^1_2=DF=w^2_2F',w^1_3=D^2F=w^2_3F'+(w^2_2)^2F^{\prime\prime},\ldots\qquad (F=F(w^2_1)).\]
We are however interested in internal theory. Then the reasonings are restricted to the subspace $\mathbf M$ and the ambient jet space $\mathbf M(2)$ is neglected.
In more detail, we introduce coordinates
\[x,\ w^1_0,\ w^2_s\qquad (s=0,1,\ldots\,)\]
on $\mathbf M$ and diffiety $\Omega\subset\Phi(\mathbf M)$ with the natural basis
\[\omega^1_0=\mbox{d}w^1_0-F\mbox{d}x,\ \omega^2_s=\mbox{d}w^2_s-w^2_{s+1}\mbox{d}x\qquad (F=F(w^2_1)).\]
The total derivative\[D=\frac{\partial}{\partial x}+F\frac{\partial}{\partial w^1_0}+\sum w^2_{s+1}\frac{\partial}{\partial w^2_s}\qquad (F=F(w^2_1))\]
is induced on $\mathbf M$ by the original operator (\ref{eq3.8}).\end{example}
Dealing with this diffiety $\Omega,$ it follows easily that
\[\mathcal L_D\omega^1_0=F'\omega^2_1,\ \mathcal L_D\omega^2_0=\omega^2_1,\ \mathcal L_D(\omega^1_0-F'\omega^2_0)=-DF'\omega^2_0.\]
Assuming $DF'\neq 0$ for now, we have the standard basis
\[\pi=\pi^1_0=\omega^1_0-F'\omega^2_0,\ \pi_1=\mathcal L_D\pi=-DF'\omega^2_0,\ \pi_2=\mathcal L^2_D\pi=-DF'\omega^2_1-D^2F'\omega^2_0,\ \ldots\]
with $K=K(\Omega)=0$ and $\mu=\mu(\Omega)=1.$ Then the formula (\ref{eq3.5}) provides all variations
\begin{equation}\label{eq3.12}
Z=z\frac{\partial}{\partial x}+\sum D^sp\,\frac{\partial}{\partial\pi_s}\qquad (z=Zx,\, p=\pi(Z)) \end{equation}
where $z,p\in\mathcal F(\mathbf M)$ are arbitrary functions. Clearly
\[Zw^2_0=\omega^2_0(Z)+w^2_1\mbox{d}x(Z)=-\frac1{DF'}\pi_1(Z)+w^2_1z=-\frac1{DF'}Dp+w^2_1z,\]
\[Zw^1_0=\omega^1_0(Z)+F\mbox{d}x(Z)=p-\frac{F'}{DF'}Dp+Fz\]
whence the (rather clumsy and in fact needless) classical formula
\[Z=z\frac{\partial}{\partial x}+\left(Fz+p-\frac{Dp}{DF'}F'\right)\frac{\partial}{\partial w^1_0}+\left(w^2_1z-\frac{Dp}{DF'}\right)\frac{\partial}{\partial w^2_0}+\cdots\]
for the variations follows. 

The true infinitesimal transformations $Z$ moreover satisfy the identity
\[\mathcal L_Z\pi=Z\rfloor\mbox{d}\pi+\mbox{d}p=\lambda\pi=\lambda(\omega^1_0-F'\omega^2_0)\qquad (\lambda\in\mathcal F(\mathbf M))\]
where $\mbox{d}\pi=\mbox{d}x\wedge\pi_1-F''\omega^2_1\wedge\omega^2_0.$ The conditions
\[-zDF'-F''\omega^2_1(Z)+p_{w^2_0}+p_{w^1_0}F'=0,\quad F''\omega^2_0(Z)+p_{w^1_1}=0,\quad p_{w^1_s}=0\ (s>1)\]
directly follow. Since $\omega^2_0(Z)-Dp/DF'=-Dp/(F''w^2_2),$ the middle condition reads
\[-\frac1{w^2_2}\left(p_x+Fp_{w^1_0}+w^2_1p_{w^2_0}+w^2_2p_{w^2_1}\right)+p_{w^2_1}=0\]
whence $p=P(Fx-w^1_0,w^2_1x-w^2_0).$ The left--hand condition determines the coefficient $z=Zx$ for the resulting infinitesimal transformation $Z$ (not written here). All infinitesimal transformations $Z$ are obtained in explicit terms.

The remaining "singular case" where $DF'=F''w^2_2=0$ hence $F=Aw^2_1+B$ $(A,B\in\mathbb R)$ is quite simple. Then
\[\omega^1_0-F'\omega^2_0=\mbox{d}w^1_0-(Aw^2_1+B)\mbox{d}x-A(\mbox{d}w^1_0-w^1_1\mbox{d}x)=\mbox{d}(w^1_0-Aw^2_0)\]
and we have the standard basis
\[\mbox{d}t^1=\mbox{d}(w^1_0-Aw^2_0),\ \omega^2_s\ (s\geq 0),\]
hence $K=K(\Omega)=1, \mu=\mu(\Omega)=1.$ Symmetries $\mathbf m$ of diffiety $\Omega$ are the Lie's contact transformations (in the space of variables $x,w^1_0,w^1_1$) depending moreover on parameter $t^1$ and the change of $t^1.$

In this example of diffiety $\Omega$ with $\mu(\Omega)=1,$ the symmetry problem is completely resolved, see \cite[Remark 5.4]{T7}.
\begin{example}\label{Example 5} (\emph{Unsolved symmetry problem.}\/)
Assuming $\mu(\Omega)>1,$ no finite algorithm for determination of \emph{all} symmetries and even of \emph{all} infinitesimal symmetries is occuring in actual literature. Only some particular solutions can be found by available methods.\end{example}
We conclude with the equation $dw^1/dx=F(x,w^1,\ldots,w^m)$ where \mbox{$m>2.$} Let $\mathbf M$ be  the space with coordinates $x, w^1_0, w^k_s$ $(k=2,\ldots,m;$$\,s=0,1,\ldots\,)$ and $\Omega\subset\Phi(\mathbf M)$ diffiety with the basis
\[\omega^1_0=\mbox{d}w^1_0-F(x,w^1_0,\ldots,w^m_0)\mbox{d}x,\ \omega^k_s=\mbox{d}w^k_s-w^k_{s+1}\mbox{d}x\qquad (k=2,\ldots,m;\,s=0,1,\ldots\,).\]
Then
\[D=\frac{\partial}{\partial x}+F\frac{\partial}{\partial w^1_0}+\sum w^k_{s+1}\frac{\partial}{\partial w^k_s}\in\mathcal H,\ \mathcal L_D\omega^1_0=\mbox{d}F-DF\mbox{d}x=\sum F_{w^j_0}\omega^j_0\]
and it follows that  two subcases should to be distinguished.

Either $\partial F/\partial w^k_0=0$ $(k=2,\ldots,m)$ identically or $\partial F/\partial w^k_0=0$ for an~appropriate $k$ $(2\leq k\leq m).$ 
One may observe that $K=K(\Omega)=1$ and, roughly saying, we have diffiety $\Omega(m-1)$ only completed with a~parameter in the first subcase. Let us therefore mention just the second subcase in more detail.

We suppose $\partial F/\partial w^m_0\neq 0$ from now on. Let us introduce the range of indices \mbox{$1\leq j\leq m,$} $2\leq k\leq m,$ $2\leq i\leq m-1$ and the abbreviation $F^j=\partial F/\partial w^j_0.$
Then
\[\mathcal L_D\omega^1_0=F^1\omega^1_0+\sum F^i\omega^ i_0+F^m\omega^m_0,\quad \mathcal L_D\omega^k_s=\omega^k_{s+1}\quad (s=0,1,\ldots)\]
which implies the congruences
\[\mathcal L_D^2\omega^1_0\cong F^m\omega^m_1\ (\text{mod } \omega^j_0,\omega^i_0),\quad \mathcal L_D^3\omega^1_0\cong F^m\omega^m_2\ (\text{mod } \omega^j_0,\omega^i_0,\omega^i_1),\ \ldots \]
and it follows that the forms
\[\pi^1_s=\mathcal L_D^s\omega^1_0,\quad \pi^i_s=\mathcal L_D^s\omega^^ i_0=\omega^i_s\quad (i=2,\ldots,m-1;\,s=0,1,\ldots)\]
may be taken for a~standard basis. We have $K=K(\Omega)=0,$ $\mu=\mu(\Omega)=m-1.$ In terms of the standard basis, all variations $Z$ are given by the series
\[Z=z\frac{\partial}{\partial x}+\sum D^sp^1\frac{\partial}{\partial \pi^1_s}+\sum D^sp^i\frac{\partial}{\partial \pi^i_s}\]
where the functions $z,p^1,\ldots,p^{m-1}\in\mathcal F(\mathbf M)$ are quite arbitrary. One may also obtain the "classical" coefficients $Zw^j_0$ in terms of functions $z$ and $p.$ They follow from the trivial formulae
\[p^1=\pi^1_0(Z)=\omega^1_0(Z)=Zw^1_0-w^1_1z,\quad p^i=\omega^i_0(Z)=Zw^i_0-w^i_1z\]
together with more involved identity
\[Dp^1=\pi^1_1(Z)=(\mathcal L_D\pi^1_0)(Z)=(\mathcal L_D\omega^1_0)(Z)=F^1\omega^1_0(Z)+\sum F^i\omega^i_0(Z)+F^m\omega^m_0(Z)\]
(where $\omega^m_0(Z)=Zw^m_0-w^m_1z$) for the remaining coefficient $Zw^m_0.$

\emph{Particular} infinitesimal transformations $Z$ can be obtained if (e.g.) the additional conditions
\[\mathcal L_Z\pi^1_0=\lambda\pi^1_0+\sum\lambda_i\pi^ i_0,\quad \mathcal L_Z\pi^i_0=\mu^i\pi^i_0+\sum \mu^i_{i'}\pi^{i'}_0\quad (i,i'=2,\ldots,m-1)\]
for the variations are prescribed. Such conditions can be easily resolved if one inserts
\[\begin{array}{l}
\mathcal L_Z\pi^1_0=Z\rfloor\mbox{d}\pi^1_0+\mbox{d}\pi^1_0(Z)=Z\rfloor(\mbox{d}x\wedge\sum F^j\omega^j_0)+\mbox{d}p^1,\\
\mathcal L_Z\pi^i_0=Z\rfloor\mbox{d}\pi^i_0+\mbox{d}\pi^i_0(Z)=Z\rfloor(\mbox{d}x\wedge\pi^i_1)+\mbox{d}p^i
\end{array}\]
which immediately gives the identities
\[\begin{array}{c}
z\sum F^j\omega^j_0+\sum\dfrac{\partial p^1}{\partial w^j_s}\omega^j_s=\lambda\omega^1_0+\sum\lambda_i\omega^ i_0,\\
z\omega^i_1+\sum\dfrac{\partial p^i}{\partial w^j_s}\omega^j_s=\mu^i\omega^i_0+\sum\mu^i_{i'}\omega^{i'}_0
\end{array}\]
in terms of the original basis. This is equivalent to the system
\[zF^1+\frac{\partial p^1}{\partial w^1_0}=\lambda,\  zF^i+\frac{\partial p^1}{\partial w^i_0}=\lambda_i,\  zF^m+\frac{\partial p^1}{\partial w^m_0}=0,\]
\[\frac{\partial p^i}{\partial w^1_0}=\mu^i,\ \frac{\partial p^i}{\partial w^{i'}_0}=\lambda^i_{i'},\ \frac{\partial p^i}{\partial w^m_0}=0,\ z+\frac{\partial p^i}{\partial w^i_1}=0\]
where
\[p^1=p^1(x,w^1_0,\ldots,w^m_0),\  p^i=p^i(x,w^1_0,\ldots,w^{m-1}_0,w^i_1).\]
Since $\lambda,\lambda_i,\mu^i,\lambda^i_{i'}$ are uncertain coefficients, we have only the condition
\[\frac{\partial p^1}{\partial w^m_0}=F^m\frac{\partial p^i}{\partial w^i_1}\qquad\text{where}\qquad \frac{\partial p^i}{\partial w^i_1}=-z.\]
It follows that $z=z(x,w^1_0,\ldots,w^{m-1}_0)$ may be arbitrarily chosen and then
\[p^1=-z\cdot\int F^m\,dw^m_0+q,\quad p^i=-zw^i_1+q^i\]
where $q$ and $q^i$ are arbitrary functions of variables $x,w^1_0,\ldots,w^{m-1}_0.$

We return to the general theory.
\section{The commutative algebra mechanisms}\label{sec4}
Let $\Omega_*:\Omega_0\subset\Omega_1\subset\cdots\subset\Omega=\cup\Omega_l$ be a~given good filtration of a~diffiety $\Omega\subset\Phi(\mathbf M).$ We introduce the graded $\mathcal F$--module
\begin{equation}\label{eq4.1}\mathcal M=\text{Grad}\,\Omega_*=\mathcal M_0\oplus\mathcal M_1\oplus\cdots\qquad (\mathcal M_l=\Omega_l/\Omega_{l-1};\, \Omega_{-1}=0).\end{equation}
It is naturally equipped with $\mathcal F$--linear mappings $Z:\mathcal M\rightarrow\mathcal M$ $(Z\in\mathcal H)$ where
\begin{equation}\label{eq4.2}Z[\omega]=[\mathcal L_Z\omega]\in\mathcal M_{l+1}\qquad ([\omega]\in\mathcal M_l;\,\omega\in\Omega_l;\,l=0,1,\ldots\,)\end{equation}
and the square brackets denote the factorization (\ref{eq4.1}). Even more can be said. The inclusion $\mathcal L_\mathcal H\Omega\subset\Omega$ and the identity $\Omega(\mathcal H)=0$ imply
\[0=X(\omega(Y))=(\mathcal L_X\omega)(Y)+\omega([X,Y])=\omega([X,Y])\qquad (X,Y\in\mathcal H;\,\omega\in\Omega)\]
hence $[\mathcal H,\mathcal H]\subset\mathcal H.$ It follows that
\[(XY-YX)[\omega]=[\mathcal L_{[X,Y]}\omega]=0\in\mathcal M_{l+2}\qquad (X,Y\in\mathcal H;\,[\omega]\in\mathcal M_l)\]
and $\mathcal M$ may be therefore regarded as $\mathcal A$--module where
\begin{equation}\label{eq4.3}\mathcal A=\mathcal A_0\oplus\mathcal A_1\oplus\cdots\qquad (\mathcal A_0=\mathcal F,\mathcal A_1=\mathcal H,\mathcal A_2=\mathcal H\odot\mathcal H,\ldots\,)\end{equation}
is the $\mathcal F$--algebra of homogeneous polynomials over the $\mathcal F$--linear space $\mathcal H.$
In more detail,
\begin{equation}\label{eq4.4}\begin{array}{c}\mathcal A_r\cdot\mathcal M_l\in\mathcal M_{l+r},\quad  Z_1\cdots Z_r[\omega]=[\mathcal L_{Z_1}\cdots\mathcal L_{Z_r}\omega]\\ \\ (Z_1,\ldots,Z_r\in\mathcal H;\, [\omega]\in\mathcal M_l).\end{array}\end{equation}
At this place, the advanced mechanisms of commutative algebra can be applied and we refer to the excellent survey \cite{T12}.
However, a~kind patience of the reader is assumed for two reasons which are as follows.
\begin{remark}\label{r4.1} 
We are interested in \emph{homogeneous} commutative algebra: only homogeneous submodules of $\mathcal M$ and in particular homogeneous ideals of algebra $\mathcal A$ make a~good sense in our theory. The common textbooks are invented for quite other aims and the general concepts and main achievements \cite {T12} need  some slight adaptation here. However, all our reasonings will be of quite simple nature and can be directly verified, see also \cite{T6}. \end{remark}
\begin{remark}\label{r4.2}
We deal with $\mathcal F$--algebra $\mathcal A$ while the common textbooks concern the polynomial algebra over a~\emph{field}. Therefore the common results can be rigorously applied only at a~fixed point $\mathbf P\in\mathbf M$ where the structural ring $\mathcal F$ turns into its localization $\mathcal F_\mathbf P=\mathbb R.$ However, we postulate the existence of $\mathcal F$--bases and $\mathcal A$--bases in all modules under consideration. (Otherwise no explicit algorithm of calculations can be performed.) It follows that then the "behaviour at $\mathbf P$" can be locally extended or, alternatively saying, the classical commutative algebra can be applied to $\mathcal F$-- and $\mathcal A$--modules as well.
\end{remark}
The following two results serve for a~transparent example.
\begin{theorem}\label{t4.1}
Every $\mathcal A$--submodule $\mathcal N\subset\mathcal M$ is finitely $\mathcal A$--generated.
\end{theorem}
This is the familiar \emph{Hilbert basis theorem}. In our theory, we suppose even the existence of a~finite $\mathcal A$--basis of $\mathcal N.$
\begin{theorem}\label{t4.2}
If $\mathcal M$ is regarded as $\mathcal F$-- module, then
\begin{equation}\label{eq4.5}
\dim\mathcal M_l=e_\nu\binom l\nu+\cdots + e_0\binom l0\qquad (\,l \text{ large enough})  \end{equation}
is the Hilbert polynomial function of the variable $l.$
\end{theorem}

We recall that the use of the combinatorial factors $\binom lk=l!/(k!(l-k)!)$ ensures the \emph{integer coefficients} $e_\nu,\ldots,e_0.$ Assuming $e_\nu\neq 0,$ we may even denote
\begin{equation}\label{eq4.6} \nu=\nu(\Omega)\neq 0,\ \mu=\mu(\Omega)=e_\nu \end{equation}
since these values \emph{do not depend} on the choice of the primary filtrations $\Omega_*,$ see \cite{T6} for easy proof. 
Clearly $\nu(\Omega)\leq n(\Omega)-1,$ $1\leq \mu(\Omega)$ and in accordance with the theory of the exterior differential systems \cite{T13, T14, T15} we may (a~somewhat formally) declare that \emph{the solutions of diffiety $\Omega$ depend on} (better: \emph{can be parametrized by}) $\mu(\Omega)$ \emph{arbitrary functions of $\nu(\Omega)+1$ variables.} In the "degenerate" case of a~finite--dimensional underlying space $\mathbf M,$ the Hilbert polynomial vanishes and we put $\nu(\Omega)=-1$ and $\mu(\Omega)=\dim \Phi/\Omega$ (apply the Frobenius theorem).

The following topics are not currently investigated in literature. They were initiated by a~brief notice \cite{T16} and thoroughly discussed in \cite{T6}.
We preserve the same notation of \mbox{$\mathcal F$--modules} and $\mathcal A$--algebra as before. More rigorously, the reasonings should be "localized" at a~fixed point $\mathbf P\in\mathbf M$ and this would provide the $\mathbb R$--linear spaces and polynomials over~$\mathbb R$ in better accordance with the common algebra.

Let $Z_1,\ldots,Z_n$ be a~given basis of $\mathcal F$--module $\mathcal H$ and $\mathcal A(i)\subset\mathcal A$ the ideal generated by $Z_1,\ldots,Z_i$ where $0\leq i\leq n$ is a~fixed integer. In particular $\mathcal A(0)=0$ and
\[ \mathcal A(n)=\mathcal A_1\oplus\mathcal A_2\oplus\cdots =\mathfrak m \]
is the (so--called \emph{improper}) \emph{maximal ideal} of algebra $\mathcal A.$ Let us introduce the graded factormodules
\[\mathcal M(i)=\mathcal M/\mathcal A(i)\mathcal M=\mathcal M(i)_0\oplus\mathcal M(i)_1\oplus\cdots \quad (\mathcal M(i)_l=\mathcal M_l/\mathcal A(i)\mathcal M\cap\mathcal M_l).\]
In particular $\mathcal M(0)=\mathcal M$ and $\mathcal M(0)_l=\mathcal M_l.$ They are $\mathcal A$--modules as well and we may consider multiplication mappings
\begin{equation}\label{eq4.7} Z_{i+1}: \mathcal M(i)_l\rightarrow \mathcal M(i)_{l+1}\quad (i=0,\ldots, n-1;\, l=0,1,\ldots \,) \end{equation}
which are of the highest importance in many respects.
\begin{definition}\rm
\label{def4.1}\label{def4.1}\rm
 Basis $Z_1,\ldots, Z_n$ of the module $\mathcal H$ is called \emph{regular} if (\ref{eq4.7}) are injective mappings for all $l\geq 0,$ \emph{quasiregular} if (\ref{eq4.7}) are injective for all $l\geq 1$ and \emph{ordinary} (in better accordance with literature \cite{T15},
or \emph{generic} \cite{T6}
) if (\ref{eq4.7}) are injective mappings for all $l$ large enough. \end{definition}

Our crucial result reads as follows.

\begin{theorem}
\label{t4.3a} The ordinary basis of the module $\mathcal H$ exists.
\end{theorem}

In order to prove this "simple" assertion, a~slight reformulation is useful. 
We recall the $c$--lift $\Omega_{*+c}$ $(c=0,1,\ldots)$ of the filtration $\Omega_*$ from Remark \ref{rem2.1}.  It is defined by
\begin{equation}\label{eq4.8} \Omega_{*+c}=\bar\Omega_*:\bar\Omega_0=\Omega_c\subset\bar\Omega_1=\Omega_{c+1}\subset\cdots\subset\Omega=\cup\bar\Omega_l.
 \end{equation}
Roughly saying, the lift corresponds to the classical concept of the prolongation of the "initial" Pfaffian system $\omega=0$ $(\omega\in\Omega_0),$ see Appendix. In classical theory, the prolongation procedure is rather involved since it starts with the "initial" submodule $\Omega_0\subset\Omega$ and the final result $\Omega$ appears after lengthy calculations of "regular" integral elements \cite{T15}.
In our approach, the prolongation is expressed by simple requirement (\ref{eq2.3}) which \emph{does not} exclude the "singular solutions" and "partial prolongations".
Therefore the following "prolongation to involutiveness" result with very clear proof proposed in \cite[point $(\kappa\iota\iota)$ on page 136]{T6} is worth attention.
\begin{definition}\label{def4.2}\rm
Filtration $\Omega_*$  is called \emph{involutive} if there exists a~quasiregular basis and moreover $\Omega_l+\mathcal L_\mathcal H\Omega_l=\Omega_{l+1}$ (equivalently $\mathcal H\mathcal M_l=\mathcal M_{l+1}$) for $l\geq 0.$
\end{definition}
We may state the common reformulation of the latter Theorem~\ref{t4.3a}.
\begin{theorem}
\label{t4.3} Filtration $\Omega_{*+c}$ is involutive if $c$ is large enough.
\end{theorem}
\begin{remark}\label{r4.3}
Before passing to the proof, let us recall some concepts from commutative algebra. They concern the graded ideals of polynomial algebra $\mathcal A$ and the graded \mbox{$\mathcal A$--module} $\mathcal M.$
A~proper $\mathcal A$--submodule $\mathfrak p\subset\mathcal A$ of algebra $\mathcal A$ is called an~\emph{ideal} $\mathfrak p$ of algebra $\mathcal A,$  so we \emph{exclude} $\mathfrak p=\mathcal A$ but the zeroth ideal $\mathfrak p=0$ is admitted.
An~ideal $\mathfrak p$ is \emph{prime} if $u\cdot v\in\mathfrak p$ $(u,v\in\mathcal A)$ implies either $u\in\mathfrak p$ or $v\in\mathfrak p.$ For every $\mathcal A$--module $\mathcal M,$ $\text{Ann}\,\mathcal M\subset\mathcal A$ is the ideal including all $u\in\mathcal A$ such that $u\mathcal M=0$ and the ideal $\text{Nil}\,\mathcal M\subset\mathcal A$ includes all $u\in\mathcal A$ such that $u^r\mathcal M=0$ for $r$ large enough. (These ideals can be related to the \emph{Cauchy characteristics}, see \cite[VII 6]{T6}.)
For every $\mathcal M,$ $\text{Ass}\,\mathcal M$ is the set of all prime ideals $\mathfrak p$ such that there exists the submodule of $\mathcal M$ isomorphic to $\mathcal A/\mathfrak p.$ If $\mathcal M\neq 0$ is nontrivial, then $\text{Ass}\,\mathcal M$ \emph{is a~finite and nonempty set} (which is stated without proof here). We also recall the maximal ideal $\mathfrak m$ where trivially $\mathfrak m\cap\mathcal H=\mathcal A_0=\mathcal H$ but otherwise $\mathfrak p\cap\mathcal H\subset\mathcal H$ is a~\emph{proper} $\mathcal F$--linear subspace for any ideal $\mathfrak p\neq\mathfrak m.$
\end{remark}
 After this preparation, we turn to the proof which consists of a~few short steps.
\begin{proof}
Let us consider the submodule $\mathcal A[\omega]\subset\mathcal M$ for a~certain nonvanishing $[\omega]\in\mathcal M.$ Let $\mathfrak p$ be the maximal element in the set of all ideals $\text{Ann}\,\mathcal N$ for all submodules $\mathcal N\subset\mathcal A[\omega].$ Then ideal $\mathfrak p$ is prime and $\mathcal N$ isomorphic to $\mathcal A/\mathfrak p.$ Since $\mathcal N\subset\mathcal M$ is a~submodule, it follows that $\mathfrak p\in\mbox{Ass}\,\mathcal M.$

Let $u\in\mathcal A.$ Multiplication $u:\mathcal M\rightarrow\mathcal M$ is injective if and only if
\[u\notin\cup\,\text{Ann}\,\mathcal A[\omega]\quad (\text{all }[\omega]\in\mathcal M,\, [\omega]\neq 0).\]
However we have seen that
\[\cup\,\text{Ann}\,\mathcal A[\omega]\subset\cup\,\mathfrak p\quad (\text{all }\mathfrak p\in\,\text{Ass}\,\mathcal M)\]
hence $u\notin\cup\,\mathfrak p$ ensures the injectivity. If in particular $u=Z\in\mathcal H\subset\mathcal A$ then $Z:\mathcal M\rightarrow\mathcal M$ is injective if and only if
\begin{equation}\label{eq4.9} Z\notin\cup\,\mathfrak p\cap\mathcal H\quad (\text{all }\mathfrak p\in\,\text{Ass}\,\mathcal M). \end{equation}
It follows that such $Z$ exists if and only if  $\mathfrak m\notin\text{Ass}\,\mathcal M.$

The last condition is satisfied if the module $\mathcal M$ is replaced by the submodule $\mathcal M_{c+}=\mathcal M_c\oplus\mathcal M_{c+1}\oplus\cdots\subset\mathcal M.$ Indeed, $\mathfrak m\in\,\text{Ass}\,\mathcal M$ means that there exist submodules of $\mathcal M$ isomophic to $\mathcal A/\mathfrak m=\mathbb R.$ All such submodules are moreover lying in $\mathcal M_0\oplus\cdots\oplus\mathcal M_{c-1}$ if $c$ is large enough since they together generate a~finite--dimensional $\mathcal F$--submodule of $\mathcal M.$

Summarizing, there exist 
\[Z\notin\cup\,\mathfrak p\quad (\mathfrak p\in\text{Ass}\,\mathcal M_{c+},\, c\text{ large})\] 
with the injective mapping $Z:\mathcal M_{c+}\rightarrow\mathcal M_{c+}$ hence injective 
\[Z:\mathcal M_l\rightarrow\mathcal M_{l+1}\quad (l\text{ large enough}).\]
Since $\mathcal M_l=\mathcal M(0)_l,$ we obtain the mappings (\ref{eq4.7}) with $i=0$ and $Z_1=Z.$ 

Remaining conditions (\ref{eq4.7}) with $i=1\ldots,n-1$ can be discussed by using the same arguments successively applied to the factormodules $\mathcal M(1),\ldots,\mathcal M(n-1)$ instead of the $\mathcal A$--module $\mathcal M=\mathcal M(0).$
\end{proof}
\begin{remark}\label{r4.4}
It can be proved that the sets $\text{Ass}\,\mathcal M, \text{Ass}\,\mathcal M_{c+}$ of the prime ideals differ each from the other only in the presence of the ideal $\frak m.$ It follows that the conditions
\begin{equation}\label{eq4.10} 
Z_{i+1}\notin\cup\,\frak p\cap\mathcal H\quad (\frak p\in\,\text{Ass}\,\mathcal M(i),\,\frak p\neq\frak m,\,i=0,\ldots,m-1\,) 
\end{equation}
determine all ordinary sequences $Z_1,\ldots,Z_n\in\mathcal H.$
So the terms $Z_1,\ldots,Z_n$ of the ordinary basis should not belong to a~certain finite family of proper linear subspaces of module $\mathcal H.$
\end{remark}
Our next aim is to modify the lower--order terms $\Omega_l$ of a~given filtration $\Omega_*$ in such a~manner that the \emph{ordinary basis} of Theorem~\ref{t4.3a} turns into the \emph{regular basis} for the adapted filtration $\bar\Omega_*.$ Alternatively saying, we may take the lift $\Omega_{*+c}$ with the \emph{quasiregular basis} and the initial term $\Omega_c$ of the lift should be appropriately modified in order to get a~\emph{regular basis}. This is possible for the \emph{controllable diffiety} $\Omega,$ if certain obstacles $\mathcal R^a\subset\Omega$ (submodules; $a=0,\ldots,\nu(\Omega)$) are vanishing.
\begin{remark}\label{r4.5} We leave the pure algebra from now on. The above results expressed in terms of $Z$--multiplication in $\mathcal A$--modules $\mathcal M$ and various injectivity requirements will be reinterpreted by using Lie derivatives $\mathcal L_Z$ acting on filtrations $\Omega_*$ and various Ker--concepts. Instead of vector fields satisfying conditions like (\ref{eq4.9}) or (\ref{eq4.10}), we shall briefly speak of "not too special" vector fields $Z.$ Then the corresponding $\text{Ker}_Z$--modules are of the least possible dimension.
\end{remark}
\section{Standard filtrations}\label{sec5}
We are passing to the lengthy reconstruction of the lower--order terms $\Omega_l$ of a~given good filtration $\Omega_*$ of a~diffiety $\Omega.$ This is equivalent to the reconstruction of the initial terms of any lift $\Omega_{*+c}.$ So we may suppose that $\Omega_*$ is involutive filtration without any loss of generality. The calculations proper are of independent interests. Intermediate results are stated only at the appropriate places in a~convenient time. Recall that our aim is to obtain a~regular basis of module $\mathcal H$ and this is possible "modulo certain obstacles $\mathcal R^a$" which are however of the highest importance too.
\begin{definition}
\label{def5.1}\rm
For any submodule $\Theta\subset\Phi=\Phi (\mathbf M)$ and a vector field $X\in\Theta^\perp,$ let $Ker_{X}\Theta\subset\Theta$ be the submodule of all $\vartheta\in\Theta$ with $\mathcal{L}_{X}\vartheta\in\Theta.$
\end{definition}
\subsection*{Our first task.}
\emph{A given involutory filtration $\Omega_*$ will be adapted to ensure the property $Ker_X\bar\Omega_{l+1}=\bar\Omega_l$ to the maximal possible extent.} (See the Figure~2.
\begin{center}\begin{picture}(300,90)(0,-15)
\put(10,50){\line(1,0){10}}\put(20,50){\line(0,-1){50}}\put(10,55){$\Omega_0$}
\qbezier(10,35)(20,45)(34,35)\put(35,34){\vector(1,-1){1}}\qbezier(10,15)(20,30)(35,15)\put(36,14){\vector(1,-1){1}}
\put(40,50){\line(1,0){10}}\put(50,50){\line(0,-1){50}}\put(40,55){$\Omega_1$}
\qbezier(40,35)(50,45)(64,35)\put(65,34){\vector(1,-1){1}}\qbezier(40,15)(50,30)(63,30)\put(65,30){\vector(1,0){1}}
\put(70,50){\line(1,0){10}}\put(80,50){\line(0,-1){50}}\put(70,55){$\Omega_2$}
\qbezier(68,32)(80,45)(94,35)\put(95,34){\vector(1,-1){1}\put(-10,-32){$\mathcal L_X$}}
\qbezier(70,15)(80,20)(93,15)\put(95,13){\vector(1,-1){1}}
\put(10,-15){the original filtration}\put(95,55){$\ldots$}
\put(160,48){\line(1,0){10}}
\multiput(170,0)(0,5){10}{\put(0,0){\line(0,1){3}}}
\put(160,55){$\mathcal R^0$}
\qbezier(150,30)(165,50)(165,30)\qbezier(165,30)(165,10)(150,26)\put(150,26){\vector(-1,-2){1}}
\put(150,10){$\mathcal L_X$}
\put(190,50){\line(1,0){10}}\put(200,50){\line(0,-1){50}}\put(190,55){$\bar\Omega_0$}
\qbezier(190,35)(200,45)(214,35)\put(215,34){\vector(1,-1){1}}
\qbezier(220,35)(230,45)(244,35)\put(245,34){\vector(1,-1){1}}
\qbezier(250,35)(260,45)(274,35)\put(275,34){\vector(1,-1){1}}
\qbezier(250,15)(260,25)(274,15)\put(275,14){\vector(1,-1){1}}\put(263,3){$\mathcal L_X$}
\put(220,50){\line(1,0){10}}\put(230,50){\line(0,-1){50}}\put(220,55){$\bar\Omega_1$}
\put(250,50){\line(1,0){10}}\put(260,50){\line(0,-1){50}}\put(250,55){$\bar\Omega_2$}
\put(160,-15){the adapted filtration}\put(275,55){$\ldots$}
\put(115,-35){Figure 2}
\end{picture}\end{center}
\vspace{1 cm}
The last term $\mathcal R^0=Ker_X\bar\Omega_0$ will be independent of the choice of the original filtration $\Omega_*.$)
Let us turn to more detail.

Taking a not too special vector field $X\in\mathcal{H},$ we have $Ker_{X}\Omega_{l+1}=\Omega_{l}$ $(l\geq 0).$
(Use the injectivity $Z_1:\mathcal M_l\rightarrow\mathcal M_{l+1}$ $(l\geq 1)$ with $Z_1=X.$) This provides the {\it strongly descending chain} of submodules
\begin{equation}\label{ex5.1}
\cdots\supset\Omega_{1}\supset\Omega_{0}\supset Ker_{X}\Omega_{0}\supset\cdots
\supset (Ker_{X})^{K}\Omega_{0}\quad (K\geq 1)
\end{equation}
which necessarily terminates with the stationarity $(Ker_{X})^{K+1}\Omega_{0}=(Ker_{X})^{K}\Omega_{0}.$ The change of notation gives the {\it strongly ascending} filtration
\begin{equation}\label{ex5.2}
\bar\Omega_{*}:\bar\Omega_{0}\subset\bar\Omega_1\subset\cdots\subset\Omega=\cup \bar\Omega_l,
\end{equation}
$$\bar\Omega_{0}=(Ker_{X})^{K-1}\Omega_{0}, \bar\Omega_1=(Ker_{X})^{K-2}\Omega_0,\ldots ,\bar\Omega_{K-2}=Ker_{X}\Omega_0,$$ $$\bar\Omega_{K-1}=\Omega_0,\bar\Omega_{K-2}=\Omega_1,\ldots$$ of $\Omega.$
It has the obvious properties
\begin{equation}
Ker_{X}\bar\Omega_{l+1}=\bar\Omega_{l}\quad (l\geq 0),\quad
Ker_{X}\bar\Omega_{0}=(Ker_{X})^{2}\bar\Omega_{0}\neq \bar\Omega_{0}
\end{equation}
and moreover the less obvious properties \cite[VII 16]{T6} as follows.
\begin{prop}\label{Proposition5.1}
The filtration $(\ref{ex5.2})$ does not depend on the vector field $X.$
\end{prop}
\begin{proof}
Let $Z\in\mathcal{H}$ be another not too special vector field. Clearly
\begin{equation}\label{ex5.4}
Ker_{X}\bar\Omega_{l}=\bar\Omega_{l-1}=Ker_{Z}\bar\Omega_{l}\quad (l\geq K)
\end{equation}
hold true. We apply the descending induction to prove the equalities (\ref{ex5.4}) for all $l$ and start with the first step $l=K-1.$
Let us (on the contrary) assume that
$$Ker_{X}\bar\Omega_{K-1}=\bar\Omega_{K-2}\neq Ker_{Z}\bar\Omega_{K-1}$$ and that (for certainty) the inclusion $\bar\Omega_{K-2}\subset Ker_{Z}\bar\Omega_{K-1}$ is \emph{not true} so that there exists 
\[\omega\in \bar\Omega_{K-2},\quad \omega \not\in Ker_{Z}\bar\Omega_{K-1}.\]
\begin{center}\begin{picture}(250,125)(0,-30)
\put(10,10){\line(1,0){30}}\put(-10,15){$Ker_Z\bar\Omega_{K-1}$}
\put(-20,80){$Ker_X\bar\Omega_{K-1}=\bar\Omega_{K-2}$}
\qbezier[100](40,10)(60,40)(80,70)
\put(45,75){\line(1,0){20}}\put(65,75){\line(0,-1){80}}
\put(95,75){\line(1,0){20}}\put(115,75){\line(0,-1){80}}
\put(145,75){\line(1,0){20}}\put(165,75){\line(0,-1){80}}
\put(195,75){\line(1,0){20}}\put(215,75){\line(0,-1){80}}
\put(95,80){$\bar\Omega_{K-1}$}\put(145,80){$\bar\Omega_{K}$}\put(195,80){$\bar\Omega_{K+1}$}\put(245,80){$\dots$}
\put(56,17){\circle{10}}\put(53,15){$\omega$}
\qbezier(50,55)(70,65)(90,55)\put(91,54){\vector(1,-1){1}}\qbezier(100,55)(120,65)(140,55)\put(141,54){\vector(1,-1){1}}
\qbezier(150,55)(170,65)(190,55)\put(191,54){\vector(1,-1){1}}
\qbezier(150,55)(170,45)(190,50)\put(189,51){\vector(-1,1){1}}
\qbezier(200,55)(220,65)(240,55)\put(241,54){\vector(1,-1){1}}
\qbezier(200,55)(220,45)(240,50)\put(239,51){\vector(-1,1){1}}
\put(243,57){$\mathcal L_X$}\put(243,43){$\mathcal L_Z$}
\qbezier(59,21)(70,34)(90,24)\put(91,23){\vector(1,-1){1}}
\qbezier(100,25)(120,15)(140,25)\put(140,25){\vector(2,1){1}}\put(141,25){$\otimes$}
\qbezier(59,13)(100,0)(140,14)\put(140,14){\vector(2,1){1}}
\qbezier(143,15)(170,25)(190,15)\put(191,14){\vector(1,-1){1}}\put(192,7){$\otimes$}
\put(127,3){$\mathcal L_Z$}\put(177,21){$\mathcal L_X$}
\put(107,-25){Figure~3}
\end{picture}\end{center}
Then  $\mathcal{L}_{Z}\omega\in\bar\Omega_{K},$ $\mathcal{L}_{Z}\omega\not\in\bar\Omega_{K-1}.$
Therefore \[\mathcal{L}_{X}\mathcal{L}_{Z}\omega\in\bar\Omega_{K+1},  \mathcal{L}_{X}\mathcal{L}_{Z}\omega\not\in\bar\Omega_{K}.\] On the other hand $\mathcal{L}_{X}\omega\in\bar\Omega_{K-1}$ hence \[\mathcal{L}_{Z}\mathcal{L}_{X}\omega\in\bar\Omega_{K}.\] Altogether we conclude that
\[\mathcal{L}_{X}\mathcal{L}_{Z}\omega - \mathcal{L}_{Z}\mathcal{L}_{X}\omega =\mathcal{L}_{[X,Z]}\omega\not\in\bar\Omega_{K}\]
which is a~contradiction since (trivially) $\omega\in\bar\Omega_{K-1}$ and therefore $\mathcal L_{[X,Z]}\omega\in\bar\Omega_K$ by applying the induction assumption (\ref{ex5.4}).The following steps $l=K-2,$ $K-3,\ldots\,$ of the induction are analogous.
\end{proof}
\begin{prop}\label{Proposition5.2}
Filtration $(\ref{ex5.2})$ is good.
\end{prop}
\begin{proof} This assertion is trivial since $\mathcal L_Z\bar\Omega_l\subset\bar\Omega_{l+1}$ hence $\mathcal L_{\mathcal H}\bar\Omega_l\subset\Omega_{l+1}$
due to Proposition~\ref{Proposition5.1}.
\end{proof}
\begin{prop}\label{Proposition5.3}
The module
\begin{equation}\label{ex5.5}\mathcal R^0=Ker_X\bar\Omega_0=(Ker_X)^K\Omega_0=\cap(Ker_\mathcal H)^k\Omega_0\end{equation}
is generated by all differentials \emph{$\mbox{d}f\in\Phi$} lying moreover in $\Omega.$
\end{prop}
\begin{proof}
The inclusion $\mathcal{L}_{Z}\mathcal{R}^{0}\subset\mathcal{R}^{0}$ implies the Frobenius condition $\mbox{d}\mathcal R^0\cong 0$ (mod $\mathcal R^0)$ but the proof is not immediate and we refer to \cite[p. 177]{T6}. Therefore the module $\mathcal R_0$ is generated by certain total differentials $\mbox{d}f$ $(f\in\mathcal F)$ where $\mbox{d}f\in\mathcal R_0\subset\Omega.$ On the other hand every differential $\mbox{d}f\in\Omega$ is lying in certain module $\bar\Omega_l$ and then (trivially) also in $Ker_X\bar\Omega_l=\bar\Omega_{l-1}$ hence in $\mathcal R_0.$
\end{proof}
\begin{prop}\label{Proposition5.4} Module $\mathcal R^0=\mathcal R^0(\Omega)$ does not depend on the choice of the filtration $\Omega_*.$
\end{prop}
\begin{proof} This is a~trivial consequence of previous Proposition~\ref{Proposition5.3}, however, direct proof may be as follows:
\begin{equation}\label{eq5.6}
\begin{array}{lll}
\mbox{dim}\,\{\omega,\mathcal{L}_{\mathcal{H}}\omega,\ldots,
(\mathcal{L}_{\mathcal{H}})^{k}\omega\}\leq const. &\mbox{if}&\omega\in\mathcal{R}^{0},\\ \\
\mbox{dim}\,\{\omega,\mathcal{L}_{\mathcal{H}}\omega,\ldots,(\mathcal{L}_{\mathcal{H}})^{k}
\omega\}\geq k+1& \mbox{if}&\omega\in\Omega, \omega\not\in\mathcal{R}^{0},
\end{array}\end{equation} see also Figure~2.
\end{proof}
\subsection*{Our second task}
Assuming the previous result
\[Ker_X\Omega_{l+1}=\Omega_l\ (l\geq 1), Ker_X\Omega_0=\mathcal R^0, Ker_X\mathcal R^0=\mathcal R^0,\]
the construction will be repeated "modulo $X$" by employing another not too special vector field $Y\in\mathcal H.$ In more detail, we introduce the filtration
\begin{equation}\label{eq5.7}
\Omega(X)_*: \Omega(X)_0\subset\Omega(X)_1\subset\cdots\subset\Omega=\cup\Omega(X)_l\quad (\Omega(X)_l=\sum(\mathcal L_X)^k\Omega_l)\end{equation}
and \emph{our aim is to ensure $Ker_Y\Omega(X)_{l+1}=\Omega(X)_l$ in the maximal possible extent by the adaptation of the lower--order terms of filtration} (\ref{eq5.7}). Let us turn to more detail.

Taking a~not too special vector field $Y\in\mathcal H,$ we have $Ker_Y\Omega(X)_{l+1}=\Omega(X)_l$ for $l$ large enough and even for $l\geq 0$ by using a~lift. (Let us recall the injectivity $Z_2:\mathcal M(1)_l\rightarrow\mathcal M(1)_{l+1}$ $(l\geq 1)$ with $Z_2=Y.$) This provides the \emph{strongly descending} chain
\begin{equation}\label{eq5.8}
\cdots\supset\Omega(X)_{1}\supset\Omega(X)_{0}\supset Ker_{Y}\Omega(X)_{0}\supset\cdots
\supset (Ker_{Y})^{K}\Omega(X)_{0}\ \ (K\geq 1)
\end{equation}
which {\it does terminate} with the stationarity $(Ker_{Y})^{K+1}\Omega(X)_{0}=(Ker_{Y})^{K}\Omega(X)_{0},$ see  Proposition~\ref{Proposition5.5} below.
So we obtain the {\it strongly ascending} filtration
\begin{equation}\label{eq5.9}\begin{array}{c}\bar\Omega(X)_{*}:\bar\Omega(X)_0\subset\bar\Omega(X)_1\subset\cdots\subset\Omega=\cup \bar\Omega(X)_l,\\ \\
\bar\Omega(X)_0=(Ker_{Y})^{K-1}\Omega(X)_0,\ldots,\bar\Omega(X)_{K-2}=Ker_{Y}\Omega(X)_0,\\ \\
\bar\Omega(X)_{K-1}=\Omega(X)_0,\bar\Omega(X)_{K-2}=\Omega(X)_1,
\end{array}\end{equation}
of $\Omega$  with the infinite--dimensional terms. It has the obvious properties 
\begin{equation}\label{eq5.10}\begin{array}{c}Ker_{Y}\bar\Omega(X)_{l+1}= \bar\Omega(X)_{l}\quad (l\geq 0),\\ \\ Ker_{Y}\Omega(X)_{0}=(Ker_{Y})^{2}\Omega(X)_{0}\neq \bar\Omega_{0}\end{array}\end{equation}
and moreover the less obvious properties as follows.
\begin{prop}\label{Proposition5.5} The stationarity holds true. \end{prop}
\begin{proof} The Hilbert polynomials provides the desired result. Denote by \[\mathcal{M}[k]=\mathcal{M}[k]_{0}\oplus\mathcal{M}[k]_{1}\oplus\cdots,\quad \mathcal{M}[k]_{l}=\bar\Omega(X)_{k}\cap\bar\Omega_{l}/\bar\Omega(X)_{k}\cap\bar\Omega_{l-1}\]
the graded $\mathcal{F}$--module equipped moreover with $X$--multiplication.Then $\mbox{dim}\,\mathcal{M}[k]_{l}$ $=e_{0}(k)$ is the zeroth--order polynomial in $l$ (a~constant depending on $k$) and obviously $e_{0}(k+1)>e_{0}(k)$ at the places of proper inclusions in (\ref{eq5.8}). This causes the finite length.\end{proof}
\begin{prop}\label{Proposition5.6} Filtration $(\ref{eq5.9})$ does not depend on the vector field $Y.$
\end{prop}
\begin{proof} The method of Proposition~\ref{Proposition5.1} applies without much change. Modules $\bar\Omega(X)_l$ stand for the previous $\bar\Omega_l$ and the new vector field $Y$ undertakes the role of $X.$ \end{proof}
\begin{prop}\label{Proposition5.7}  Filtration $(\ref{eq5.9})$ is good. \end{prop}
\begin{proof} Conditions (\ref{eq2.3}) with $\Omega(X)_l$ instead of $\Omega_l$ are obviously satisfied due to Proposition~\ref{Proposition5.6}.
\end{proof}
\begin{prop}\label{Proposition5.8} The module
\begin{equation}\label{eq5.11} \mathcal{R}^{1}=Ker_{Y}\bar\Omega(X)_{0}=(Ker_Y)^K\Omega(X)_0\end{equation}
does not depend on the choice of the filtration $\Omega_*.$
 \end{prop}
\begin{proof} The following criteria can be applied:
\[\mbox{dim}\,\{\omega,\mathcal{L}_{\mathcal{H}}\omega,\ldots,
(\mathcal{L}_{\mathcal{H}})^{k}\omega\}\leq k\cdot const.+const.\quad  \mbox{ if }\ \omega\in\mathcal{R}^{1}, \]
\[ \mbox{dim}\,\{\omega,\mathcal{L}_{\mathcal{H}}\omega,\ldots,
(\mathcal{L}_{\mathcal{H}})^{k}\omega\}\geq \frac{k(k+1)}{2}\quad\mbox{ if }\ \omega\in\Omega, \omega\not\in\mathcal{R}^{1}. \]
In the first inequality, $\mathcal R^1$ is in fact equipped only with active $\mathcal L_X$--operators.\end{proof}
\begin{prop}\label{Proposition5.9}  Module $\mathcal R^1$ satisfies the Frobenius condition \emph{$\mbox{d}\mathcal R^1\cong 0$ $(\text{mod } \mathcal R^1).$} \end{prop}
\begin{proof} The method \cite[p. 177]{T6} can be adapted. We instead refer to the next part of this article, which will be especially devoted to all notable modules $\mathcal R^k$ in full generality. 
\end{proof}

We will not continue with the \emph{higher--order tasks} since they bring only toilsome formalisms but none of true novelties. The formulation of the most general Propositions is left to the reader. In particular, one may speak of \emph{standard filtrations} (\ref{ex5.2}) or (\ref{eq5.9}) in full accordance with the theory of ordinary differential equations \cite{T7}.
Also the \emph{standard bases} can be introduced quite analogously as in \cite{T7} but this topic will be discussed together with explicit calculation of symmetries of diffiety $\Omega.$

\section{Appendix}\label{sec6}
The involutivity is a~terrible topic in strictly rigorous expositions \cite{T15, T17}. So we supply our story with a~few informal notices on this subject. Let us compare the involutive systems of differential equations (the external theory) with the involutive Pfaffian systems (the internal theory) in the common sense and conclude with the prolongation to diffieties.

Let us begin with the \emph{external theory}.

\medskip

A~system of the first--order differential equations in the first--order jet space
\[x^ i,w^j,w^j_i\qquad(i=1,\ldots,n;\, j=1,\ldots,m;\, w^j_i=\frac{dw^j}{dx^i})\]
will be mentioned. We suppose that it looks as follows \cite{T18}. A~maximal number of equations is resolved with respect to derivatives $w^j_n.$ They are substituted into the remaining equations which are resolved with respect to maximal number of derivatives $w^j_{n-1}$ and so forth up to derivatives $w^j_1.$ We obtain
\begin{equation}\label{eq6.1}
\begin{array}{c}
w^{j_n}_n=f^{j_n}_n(x^1,\ldots,x^n,w^1_1,\ldots,w^m_{n-1},\ldots,w^{k_n}_n,\ldots)\\
(j_n\in J_n,\, k_n\in K_n),\\ \\
w^{j_{n-1}}_{n-1}=f^{j_{n-1}}_{n-1}(x^1,\ldots,x^n,w^1_1,\ldots,w^m_{n-2},\ldots,w^{k_{n-1}}_{n-1},\ldots)\\ (j_{n-1}\in J_{n-1},\, k_{n-1}\in K_{n-1}),\\
\vdots \\
w^{j_1}_1=f^{j_1}_1(x^1,\ldots,x^n,\ldots,w^{k_1}_1,\ldots)\\
(j_1\in J_1, k_1\in K_1)
\end{array}
\end{equation}
where $J_i+K_i=\{1,\ldots,n\},$ $J_i\cap K_i=\emptyset$ $(i=1,\ldots,n).$ The prolongation of (\ref{eq6.1}) consists of equations
\begin{equation}\label{eq6.2}
\frac{\text{d}}{\text{d}x^i}\left(w^{j_k}_k-f^{j_k}_k\right)=0\qquad (i,k=1,\ldots,n;\, j_k\in J_k).
\end{equation}
Even more can be said. If the independent variables are "not too special" then \mbox{$J_1\subset\cdots\subset J_n$} and therefore
$K_1\supset\cdots\supset K_n$ and the system (\ref{eq6.1}) is called \emph{involutive} if equations (\ref{eq6.2}) with $i\leq k$ are enough in prolongation (\ref{eq6.2}).
The prolonged system is of the second order, however, it may be interpreted as the first--order system if the first--order derivatives are regarded for new variables.  One can then easily find that the involutivity is preserved: if $\sigma^i$ $(\bar\sigma^i)$ is the number of elements in $K^i$ (in analogous sets $\bar K^i$ after the prolongation) then
\[\bar\sigma_1=\sigma_1+\cdots +\sigma_n,\ \bar\sigma_2=\sigma_2+\cdots +\sigma_n,\ \bar\sigma_n=\sigma_n.\]
This is the Cartan's \emph{involutivity test} \cite{T1, T14, T15, T17, T18}.

It is not  easy to prove that the involutiveness can be attained after a~finite number of prolongations and there are lengthy studies devoted to this topic \cite{T17}. It should be however noted that the "exceptional points" where the numbers $\sigma^i$ are not locally constant are omitted, that is, only the "regular solutions" are  involved.

Let us turn to the \emph{internal theory}.

\medskip

System (\ref{eq6.1}) is expressed by the Pfaffian equations
\[\omega^j=0\qquad (j=1,\ldots,m;\, \omega^j=\mbox{d}w^j-\sum w^j_i\mbox{d}x^i,\, w^j_i=f^j_i \text{ if } j\in J_i)\]
involving only derivatives $w^{k_i}_i$ $(k_i\in K_i).$ The prolongation is again expressed by the Pfaffian equations
\[\omega^j_i=0\qquad (i=1,\ldots,n;\, j\in K_i;\, \omega^j_i=\mbox{d}w^j_i-\sum w^j_{ii'}\mbox{d}x^{i'})\]
where equations $\omega^j_i=0$ should imply $\mbox{d}\omega^j_i=\sum\mbox{d}x^i\wedge\mbox{d}w^j_i=0$ whence $w^j_{ii'}=w^j_{i'i}.$ Assuming this symmetry, variables $w^j_{ii'}$ may be determined \emph{successively} in the \emph{involutive case.} (This is the classical definition \cite{T1, T14, T15} for the Pfaffian systems.) That is, we may arbitrarily choose
\[\ \ \ w^{j_1}_{11},w^{j_2}_{21},w^{j_3}_{31},\ldots,w^{j_n}_{n1} \text{ in total number } \bar\sigma^1=\sigma^1+\sigma^2+\sigma^3+\cdots +\sigma^n,\]
\[\quad w^{j_2}_{22},w^{j_3}_{32},\ldots,w^{j_n}_{n2} \text{ in total number } \bar\sigma^2=\sigma^2+\sigma^3+\cdots +\sigma^n,\]
and so on to the final $w^{j_n}_{nn}$ in total number $\bar\sigma^n=\sigma^n.$
This is again the Cartan's test as above.

It seems that nobody understands the original Cartan's proof of the combinatorical nature in \cite{T1} that the involutivity can be attained within this internal framework. It should be moreover noted that the "singular prolongations" are not included both in the internal and in the external theories.

We conclude with \emph{diffieties}.

\medskip

Every diffiety $\Omega$ can be obtained from the first term $\Omega_0$ of appropriate filtrations by successive prolongations \cite{T6} which are subjected only to the condition (\ref{eq2.3}). It follows that Theorem~\ref{t4.3} can be applied to "singular solutions" and even to all "partial prolongations" of a~Pfaffian system as well.


\begin{thebibliography}{99}

\bibitem{T1} \uppercase{Cartan E.}:
\textit{Sur la structure des groupes infinis de transformations}.
(French) \textit{Ann. Sci. \'Ecole Norm. Sup.} (3) \textbf{21} (1904), 153-206.

\bibitem{T2} \uppercase{Cartan E.}: 
\textit{Sur l'equivalence absolu de certains syst\'emes d'equations diff\'erentielles at sur certaines familles de courbes}. Bull. Soc. Math., t. \textbf{42}, p. 12-48) \textit{Oeuvr\'es Compl\'etes} II2.

\bibitem{T3} \uppercase{Krasil'shchik I. S.---Lychagin V. V.---Vinogradov A. M.}:
\textit{Geometry of jet spaces and nonlinear partial differential equations}. 
Transl. from the Russian. (English) Advanced Studies in Contemporary Mathematics, 1. New York etc.: Gordon and Breach Science Publishers, New York (1986), 441 p. 

\bibitem{T4} \uppercase{Kuznetsova M. N.---Pekcan A.---Zhiber A. V.}: 
\textit{The Klein-Gordon equation and differential substitutions of the form} $v=\varphi(u,u_x ,u_y ).$ (English); 
SIGMA, Symmetry Integrability Geom. Methods Appl. 8, Paper 090, 37 p., electronic only (2012).
MSC:  35L70 35A25 35A22

\bibitem{T5} \uppercase{Kamran N.}: 
\textit{Selected Topics in the Geometrical Study of Differential Equations}.CBMS Regional Conference Series in Mathematics, \textbf{96}. Providence, RI: American Mathematical Society (AMS) (ISBN 0-8218-2639-5/pbk). xi, (2002) 115 p.

\bibitem{T6} \uppercase{Chrastina J.}: 
\textit{The formal theory of differential equations}.
Folia Facultatis Scientiarum Naturalium Universitatis Masarykianae Brunensis. Mathematica 6. Brno: Masaryk University. x, (1998) 296 p.

\bibitem{T7} \uppercase{Tryhuk V.---Chrastinov\'a V.}: 
\textit{Automorphisms of ordinary differential equations}. Abstract and Applied Analysis, Volume 2013, Article ID 482963, 32 pages, http://dx.doi.org/10.1155/2013/482963


\bibitem{T8} \uppercase{Tryhuk V.---Chrastinov\'a V.---Dlouh\'y O.}: 
\textit{The Lie Group in Infinite Dimension}. 
Abstract and Applied Analysis, vol. 2011, Article ID 919538, 35 pages, 2011, DOI:10.1155/2011/919538.

\bibitem{T9} \uppercase{Tryhuk V.---Chrastinov\'a V.}: 
\textit{On the mapping of jet spaces}. J. Nonlinear Math. Phys., \textbf{17}, No. 3, 2010, 293-310, DOI No: 10.1142/S140292511000091X.

\bibitem{T10} \uppercase{Tryhuk V.---Chrastinov\'a V.}:
\textit{Automorphisms of curves}. J. Nonlinear Math. Phys.,
			 \textbf{16}, No. 3, 2009, 259-281.

\bibitem{T11} \uppercase{Chrastinov\'a V.}: 
\textit{Report on the higher--order contact transformations}.
7-th Conference on Mathematics and Physics on Technical Universities, Brno (2011, Sept. 22) 176-188.

\bibitem{T12} \uppercase{Eisenbud, D}: 
\textit{Commutative algebra with a~View Toward Algebraic Geometry}. Springer Verlag, 1994.

\bibitem{T13} \uppercase{Cartan, E.}:
\textit{Sur l'int\'egration des syst\'emes d'\'equations aux diff\'erentielles totales}. Ann. d. l'\'Ec. Norm. (3) \textbf{18},1901, 241-311.

\bibitem{T14} \uppercase{Cartan, E.}:
\textit{Exterior differential systemes and their geometric applications}  (\textit{Les syst\'emes diff\'erentiels exterieurs et leurs applications geom\'etriques}\/). Paris: Hermann \& Cie., 214 p. (1945); 
\textit{Les syst\'emes diff\'erentiels ext\'erieurs et leurs applications g\'eometriques. 2e ed.}.
Actualit\'es scientifiques et industrielles, No.994, Paris: Hermann, 210 p. (1971).

\bibitem{T15} \uppercase{Bryant, R.---Chern, S. S.---Goldschmidt, H.---Griffiths, P. A.}: 
\textit{Exterior differential systems}. 
Mat. Sci. Res. Inst. Publ. \textbf{18}, Springer - Verlag 1991.

\bibitem{T16} \uppercase{Serre J. P.}: 
A letter in Bulletin AMS 70, 1964, 42--46.

\bibitem{T17} \uppercase{Malgrange, B.}:
\textit{Cartan involutiveness = Mumford regularity}. 
Providence, RI: American Mathematical Society (AMS) (ISBN 0-8218-3233-6/pbk). Contemp. Math. 331 (2003), 193-205.

\bibitem{T18} \uppercase{Cartan, E}: 
\textit{Sur les \'equations de la gravitation d'Einstein}. (French),
Journ. de Math. (9) \textbf{1}  (1922), 141-203. 

\end{thebibliography}
\end{document}